\documentclass[11pt]{article}

\usepackage[T1]{fontenc}
\usepackage{amsfonts}
\usepackage{amsmath}
\usepackage{amssymb}
\usepackage{amsthm}
\usepackage{bbm}
\usepackage{bm}
\usepackage{mathrsfs}
\usepackage{verbatim}
\usepackage{setspace}
\usepackage{color}
\usepackage{pdfsync}
\usepackage{enumitem}
\usepackage{graphicx}
\usepackage{stmaryrd}
\usepackage{float} %place figure in text
\usepackage{tikz}
\usetikzlibrary{automata, positioning}
\usepackage{tcolorbox}

\usepackage[margin=1.25in]{geometry}
\theoremstyle{plain}
\newtheorem{theorem}{Theorem}[section]
\newtheorem{proposition}[theorem]{Proposition}
\newtheorem{lemma}[theorem]{Lemma}
\newtheorem{corollary}[theorem]{Corollary}

\theoremstyle{definition}
\newtheorem{definition}[theorem]{Definition}
\newtheorem{remark}[theorem]{Remark}

\newtheorem{example}[theorem]{Example}

\theoremstyle{remark}
\renewenvironment{thebibliography}[1]{%
\begin{oldthebibliography}{#1}%
\setlength{\baselineskip}{.9em}
\linespread{1}
\small
\setlength{\parskip}{0.3ex}%
\setlength{\itemsep}{.2em}%
}%
{%
\end{oldthebibliography}%
}
\newcommand{\R}{\mathbb{R}}

\newcommand{\cF}{\mathcal{F}}

\DeclareMathOperator{\BetaFun}{Beta}

\DeclareMathOperator{\Law}{Law}

\newcommand{\qforq}{\quad\mbox{for}\quad}
\newcommand{\qandq}{\quad\mbox{and}\quad}
\newcommand{\tF}{\tilde{F}}
\newcommand{\tR}{\tilde{R}}

\numberwithin{equation}{section}

\usepackage[pdfborder={0 0 0}]{hyperref}
\hypersetup{
  urlcolor = black,
  pdfauthor = {Marcel Nutz, Yuchong Zhang},
  pdfkeywords = {Stochastic Contest; Stackelberg Game},
  pdftitle = {Reward Design in Risk-Taking Contests},
  pdfsubject = {Reward Design in Risk-Taking Contests},
  pdfpagemode = UseNone
}
\begin{document}

\title{\vspace{-2.5em}
  Reward Design in Risk-Taking Contests
\date{\today}
\author{Marcel Nutz\thanks{Departments of Statistics and Mathematics, Columbia University, New York, USA, \texttt{mnutz@columbia.edu}. Research supported by an Alfred P.\ Sloan Fellowship and NSF Grant DMS-1812661.}
\and
Yuchong Zhang\thanks{Department of Statistical Sciences, University of Toronto,
Canada, \texttt{yuchong.zhang@utoronto.ca}. Research supported by NSERC Discovery Grant RGPIN-2020-06290.}
}
}
\maketitle \vspace{-1em}

\begin{abstract}
  Following the risk-taking model of Seel and Strack, $n$ players decide when to stop privately observed Brownian motions with drift and absorption at zero. They are then ranked according to their level of stopping and paid a rank-dependent reward. We study the problem of a principal who aims to induce a desirable equilibrium performance of the players by choosing how much reward is attributed to each rank. Specifically, we determine optimal reward schemes for principals interested in the average performance and the performance at a given rank. While the former can be related to reward inequality in the Lorenz sense, the latter can have a surprising shape.
\end{abstract}

\vspace{0.9em}

{\small
\noindent \emph{Keywords} Stochastic Contest; Stackelberg Game; Optimal Stopping

\noindent \emph{AMS 2020 Subject Classification}
%91B43; (2020-now) Principal-agent models
91A65; %(2000-now) Hierarchical games (including Stackelberg games)
91A15; %(2000-now) Stochastic games
91A55 %(2000-now) Games of timing
%60G40; %(1973-now) Stopping times; optimal stopping problems; gambling theory
%93E20; %(1973-now) Optimal stochastic control
%91A13; %(2000-now) Games with infinitely many players
%91A15 %(2000-now) Stochastic games
%}

%%%%%%%%%%%%%%%%%%%%%%%%%%%%%%%%%%%%%

\section{Introduction}

We consider the Seel--Strack model~\cite{SeelStrack.13} of risk-taking under private information and relative performance pay: $n$ players decide when to stop privately observed, i.i.d.\ Brownian motions with drift. As the processes are absorbed at the origin, players risk bankruptcy by gambling longer, and this risk represents a cost for stopping later. Once all players have stopped, they are rewarded according to their relative ranks. 
Seel and Strack focus on a winner-takes-all game, meaning that only the top-ranked player receives a reward and the players' problem boils down to maximizing the probability of winning. Here, we consider arbitrary reward schemes where subsequent ranks may also receive payments. For instance, a hedge fund may compensate managers according to their rank, giving smaller bonuses also to the second and third-best performers, or even to all managers. Or, a firm may decide on promotions and terminations based on relative performance.
 The game admits a unique Nash equilibrium for any reward scheme.

A main result of Seel and Strack was that their contest is an inappropriate compensation scheme for firms because even a small negative drift can lead to large losses in the performance of an average manager---as the players care only about their relative ranking and not the absolute level of stopping, the winner-takes-all design induces risk-seeking behavior and the associated  extended gambling implies that the drift takes a significant toll on the average performance.
This observation is a motivation for our investigation: how should a principal allocate a given reward budget over the ranks in order to incentivize a desirable performance (stopping level) by the agents in equilibrium? This Stackelberg game is studied for several objective functions.  %
Mathematically, reward inequality in the sense of Lorenz order leads to a single-crossing property of the stopping distributions which drives several of our results.

First, we show that a principal deriving utility from the performance of the average player can use the reward design to align agents' risk preferences with her own, under suitable market conditions. Under negative drift, a risk-averse principal benefits from a more equal compensation scheme. Indeed, this alleviates the issue raised in~\cite{SeelStrack.13}: as players are less incentivized to gamble and stop sooner, their performance suffers less from the declining market. While as in~\cite{SeelStrack.13}, the largest losses still occur for small negative values of the drift, their magnitude is greatly reduced. Under positive drift, there is a trade-off between risk aversion and benefit from mean return, which results in an ambiguous comparison.

Second, we study a principal maximizing the expected performance of the first-ranked player. For instance, a firm launching a competition for a novel product design or architecture project may be interested in the winning submission (that will be realized later on) rather than the average.
The performance of the first-ranked player is shown to be monotone in Lorenz order for any market condition, and as a result, the winner-takes-all scheme is always optimal. Intuitively, this principal reaps outsized benefits from higher variance in the performance distribution which outweigh possible losses from a negative drift over time.

Third, we consider a principal maximizing the expected performance at the $k$-th rank, where $1< k\leq n-1$. As an example, consider a platform linking buyers and sellers
in sealed-bid, second-price auctions (as common e.g.\ in online advertising). If the platform receives a percentage of the price paid (i.e., the second-highest bid) and develops a reward program for bidders based on ranks, how should a given budget be distributed?
 A first guess may be to give equal rewards to the first two ranks. More generally we may consider the cut-off scheme at rank $j$, which allocates equal rewards to the first $j$ ranks and nothing to the rest---for instance, a company distinguishing franchises with a top-ten award or promoting its five best-performing employees (or terminating the worst-performing employees, as relevant to the fund industry~\cite{KempfEtAl.09}). The performance at the $k$-th rank turns out to be more subtle than the first rank. Indeed, the benefits from variance decline as $k$ increases, and other effects come to play. Under zero drift, a cutoff at rank 2 is optimal for the second-rank performance, but this result does not extend to larger $k$: while a cut-off is still optimal, it can be preferable to attribute rewards beyond the $k$-th rank. For example, when $n=10$, the performance of the median player ($k=5$) is optimized by paying equal rewards to the first 7 ranks. For positive drift, cutoff schemes are again optimal, whereas for negative drift, the optimal scheme may also pay an intermediate amount.
 
The winner-takes-all contest of~\cite{SeelStrack.13} has been extended in several directions, including more general diffusion processes~\cite{FengHobson.15}, random initial laws~\cite{FengHobson.16a}, heterogeneous loss constraints~\cite{Seel.15} and a behavioral model~\cite{FengHobson.16b} where losers may be penalized if they (a posteriori) missed an opportunity to win. A different model~\cite{SeelStrack.16} has no bankruptcy condition but instead postulates a flow cost that is charged until stopping. 
Rank-order prize allocations have been studied extensively for static games; see \cite[Chapter~3]{Vojnovic.2016} for an introduction and related literature. 
In the game of~\cite{FangNoe.16}, players independently choose any distribution on~$\R_{+}$ subject to an upper bound on the mean and receive rank-based rewards according to their realization. The authors establish existence and uniqueness of an equilibrium and show, among other comparative statics, that reward inequality leads to greater dispersion of the equilibrium distribution in the sense of convex order.
In a different but related model with convex effort costs~\cite{FangNoeStrack.20}, reward inequality is shown to decrease efforts. The authors discuss the implications of this ``discouragement effect'' in numerous areas such as managerial compensation, employee promotion, grading and admissions in higher education. Many of their conclusions are also relevant to the present paper.

Via Skorokhod's embedding theorem, the game of~\cite{FangNoe.16} is equivalent to the present timing game in the case of driftless Brownian motion. When the drift is nonzero, a monotone transformation can be used to identify equilibria with the driftless case. As rewards only depend on ranks and ranks are preserved by the transformation, this immediately implies the existence and uniqueness of an equilibrium. On the other hand, comparative statics that are not invariant under monotone transformations may differ---for instance, the aforementioned result on dispersion does not hold for positive drift (Example~\ref{ex:principal-bull-market}). The main difference with the present study, however, is our focus on a principal designing the reward. To the best of our knowledge, performance at a given rank has not been studied in these games.

Related but different rank-based games have been studied in \cite{BayCviZhang.19,BayZhang.19, NutzZhang.19}. In a dynamic Poissonian game where players control the jump intensity and are ranked according to their jump times, \cite{NutzZhang.19} shows that the expected jump time of the $k$-th ranked player is minimized by a reward scheme which pays nothing to the ranks below~$k$. The amounts paid to ranks $1, \ldots, k$ are positive and strictly concave; in particular, unlike in the present model, they are not equal.  In the mean field game limit with an infinite number of competing players, the effect of reward inequality and several contest design problems are analyzed in \cite{BayZhang.19} and \cite{BayCviZhang.19}. Here players exert effort to maximize rewards based on the ranking of their terminal position and completion time of drifted Brownian motions, respectively, but analytical results are not available for the associated finite-player games.

Following this Introduction, Section~\ref{se:equilibrium} details the model and the equilibrium for a given reward scheme, whereas Section~\ref{se:design} studies the optimal reward design.

\section{Equilibrium}\label{se:equilibrium}

We fix the number $n\geq 2$ of players. For $1\leq i\leq n$, consider a diffusion $X^{i}_{t}=x_{0}+\mu t + \sigma W^{i}_{t}$ with absorption at $x=0$. The parameters $x_{0},\sigma\in(0,\infty)$ and $\mu\in\R$ are common among all players whereas the standard Brownian motions $W^{i}$ are independent. Each player~$i$ observes only her own diffusion and chooses a possibly randomized stopping time $\tau_{i}<\infty$. The players are then ranked according to the level $X^{i}_{\tau_{i}}$ at which they stopped, with ties split uniformly at random. The player with rank~$k$ is given a reward $R_{k}$. These prizes are deterministic and ordered, $R_{1}\geq R_{2} \geq \dots \geq R_{n}\geq0$,  with $R_{1}>R_{n}$ to exclude the constant case where any profile of stopping times is an equilibrium.\footnote{Ordered prizes are natural in the applications we have in mind, like employee compensation or auctions, where a different scheme may not be acceptable to players in the first place. We mention that non-monotone rewards can lead to non-existence of %finite
equilibrium stopping times or atoms in the equilibrium distribution, issues that we prefer to avoid here.} We denote the total reward by $R_{tot}:=\sum_{k=1}^{n} R_{k}$ and the average reward by $\bar R:=R_{tot}/n$.

A given stopping time $\tau_{i}$ leads to a distribution $F=\Law(X^{i}_{\tau_{i}})$ for the position at stopping. The set $\cF$ of distributions that are \emph{feasible} in this sense is readily characterized through Skorokhod's embedding theorem, as observed in~\cite{SeelStrack.13}.

\begin{lemma}\label{le:attainableDistrib}
  The set $\cF$ consists of all distributions $F$ supported on $[0,\infty)$ satisfying 
  $\int_{\R} h(x) F(dx)=1$ if $\mu> 0$ and $\int_{\R} h(x) F(dx)\le 1$ if $\mu\le 0$, respectively, where $h$ is the normalized scale function
\begin{equation}\label{eq:defh}
  h(x)=\begin{cases}
  \frac{\exp(\frac{-2\mu x}{\sigma^{2}})-1}{\exp(\frac{-2\mu x_{0}}{\sigma^{2}})-1}, & \mu\neq 0,\\
  \frac{x}{x_0}, & \mu=0.
  \end{cases}
\end{equation}
\end{lemma} 
This result goes back to \cite{Hall.69}; see~\cite[Section~9]{Obloj.04} for a systematic derivation and background. (The extension to the present case with absorbing boundary is immediate.)
We say that $F\in\cF$ is an \emph{equilibrium} distribution if, for i.i.d.\ stopping levels $X^{i}_{\tau_{i}}\sim F$, no player is incentivized to choose a different stopping time (or equivalently, a different distribution in~$\cF$).
Mathematically, let $u^F(x)$ be the expected payoff of player 1 (say) for stopping at level $x$ if all other players stop according to $F$. The probability that among players $2, \ldots, n$, there are exactly $i$ players stopping above $x$, $j$ players below $x$, and $k$ players at $x$, is given by
\[{{n-1}\choose{i,j,k}}(1-F(x))^i F(x-)^j (F(x)-F(x-))^k.\]
Here and below, we use the same symbol~$F$ to denote the measure and its cdf, and $F(x-):=\lim_{y\uparrow x}F(y)$.
Such a configuration leads to an average payoff 
$(R_{i+1}+\cdots+ R_{i+k+1})/(k+1)$
for player~$1$ as ties are broken randomly, and it follows that
\begin{equation}\label{eq:uF}
u^F(x)=\sum_{\substack{i,j,k\ge 0 \\ i+j+k= n-1}} \frac{R_{i+1}+\cdots+ R_{n-j}}{k+1} {{n-1}\choose{i,j,k}}(1-F(x))^i F(x-)^j (F(x)-F(x-))^k.
\end{equation}
Then $F\in \cF$ is an equilibrium if $\int u^F dF \ge \int u^F d\tilde F$ for all $\tilde F\in \cF$.

The equilibrium can be motivated through an ansatz as follows. We guess  that there is an equilibrium $F$ with no atoms and support $[0,\bar x]$ for some $0<\bar x<\infty$.
For $0\leq x\leq \bar x$, let $u(x)=u^F(x)$ be the expected payoff defined above.
As $F$ is atomless, $x=\bar x$ leads to the first rank with probability one, hence $u(\bar x)=R_{1}$. Similarly, $u(0)=R_{n}$, and
symmetry suggests that $  u(x_{0})=\bar R$. More generally, we guess that in equilibrium, player~1 is invariant between all stopping times $0\leq \tau \leq \bar \tau$, where $\bar \tau$ is the first exit time from $[0,\bar x]$. This translates to the condition that $u(X)$ is a martingale as long as $X$ stays within $(0,\bar x)$. 
If $u$ is smooth on $(0,\bar x)$, it follows via It\^{o}'s formula that
$\mu u'(x) + \frac12 \sigma^{2} u''(x)=0$ on that interval.
For $\mu\neq 0$, the unique function satisfying all these conditions is
\begin{equation}\label{eq:valueFun}
  u(x)= (\bar R - R_{n}) \frac{\exp(\frac{-2\mu x}{\sigma^{2}})-1}{\exp(\frac{-2\mu x_{0}}{\sigma^{2}})-1} + R_{n}, \quad 0\leq x\leq \bar x,
\end{equation}
where $\bar x$ is determined via $u(\bar x)=R_{1}$ to be 
\begin{equation}\label{eq:xbar}
  \bar x = \frac{\sigma^{2}}{-2\mu}\log\left\{\frac{R_{1}-R_{n}}{\bar R - R_{n}}\left[\exp\Big(\frac{-2\mu x_{0}}{\sigma^{2}}\Big)-1\right] + 1\right\}.
\end{equation}
More precisely, this expression is finite (and strictly positive) when $\mu<\bar\mu$, where $\bar\mu>0$ is defined by setting the argument of the above logarithm to zero,
\begin{equation}\label{eq:mubar}
  \bar\mu = \frac{\sigma^{2}}{2x_{0}} \log\left( \frac{R_{1} - R_{n}}{R_{1} -\bar R}\right).
\end{equation}
The restriction $\mu<\bar\mu$ is a \emph{standing assumption}. It ensures that players stop in finite time; in particular, the ranking is well-defined. In the driftless case $\mu=0$, the above simplifies to
\begin{equation}\label{eq:valuefunBarxMuNull}
  u(x)=\frac{\bar R - R_{n}}{x_{0}} x + R_{n}, \quad \bar x = \frac{R_{1} - R_{n}}{\bar R - R_{n}}x_{0}.
\end{equation}
On the other hand, since $F$ is atomless, \eqref{eq:uF} simplifies to 
\[u(x)=\sum_{k=1}^{n} R_{k}{{n-1}\choose{k-1}} F(x)^{n-k}(1-F(x))^{k-1}.\]
This right-hand side is of the form $g(F(x))$, and the following allows us to define~$F$ by inverting~$g$.

\begin{lemma}\label{le:solFexists}
  The function 
  $$
    g: [0,1]\to [R_{n},R_{1}], \quad g(y)=\sum_{k=1}^{n} R_{k} {{n-1}\choose{k-1}} y^{n-k}(1-y)^{k-1}
  $$
  is strictly increasing, hence invertible on $[R_{n},R_{1}]=[u(0),u(\bar x)]$. Define
  \begin{equation}\label{eq:defF}
  F(x)= g^{-1} (u(x)), \quad 0\leq x\leq \bar x
  \end{equation}
  as well as $F(x)=0$ for $x<0$ and $F(x)=1$ for $x>\bar x$. Then $F$ is the cdf of an atomless distribution with support $[0,\bar x]$ whose density~$f$ is %continuous and 
  strictly positive on $(0,\bar x)$. Moreover, $F\in\cF$.
\end{lemma} 

The stated properties of~$g$ follow from the observation that $g(y)$ is the expected reward for stopping at $y$ in the game where the other $n-1$ players stop according to a uniform distribution on $[0,1]$. A direct computation shows $\int h dF=1$, so that $F\in\cF$ is guaranteed by Lemma~\ref{le:attainableDistrib}.

The construction implies that $F$ is indeed an equilibrium: If players $2,\dots,n$ have stopping distribution $F$, then $u$ is the value function for player~1; in particular, player~1 can attain an expected reward of $\bar R$ by choosing~$F$ as well. If $\tau$ is any stopping time (possibly randomized), It\^{o}'s formula and the fact that $X:=X^{1}$ is absorbed at $0$ imply that $u(X_{t})$ is a nonnegative supermartingale and in particular $E[u(X_{\tau})]\leq u(x_{0})=\bar R$. Hence, player~1 has no incentive to deviate from~$F$, showing that $F$ is an equilibrium.

\begin{proposition}\label{pr:equilibrium}
 Let $u,\bar x,\bar\mu,F$ be defined as in \eqref{eq:valueFun}--\eqref{eq:defF} and $\mu<\bar\mu$. There exists a unique equilibrium, given by the distribution~$F$, and $u$ is the corresponding equilibrium value function.
\end{proposition} 

\begin{proof}
  In the case $\mu=0$, Lemma~\ref{le:attainableDistrib} shows that the game is equivalent to the static, capacity-constrained game of~\cite{FangNoe.16}, where players choose among all distributions~$F$ on~$\R_{+}$ with $\int x\,dF\leq x_{0}$. Existence and uniqueness is established in~\cite[Theorem~1]{FangNoe.16}. If $\mu\neq0$, using the fact that the reward is based solely on the rank as well as $\mu< \bar\mu$, we see that $F$ is an equilibrium if and only if $\tF:=F\circ h^{-1}$ is an equilibrium of the game with $\mu=0$, and the proposition follows.
\end{proof} 

\begin{remark}\label{rk:valuesUsed}
(a) The value function $u$ depends on the minimal, maximal, and average reward, but not on the further details of the reward vector~$R$. By contrast, the equilibrium distribution depends on all rewards $R_{k}$. More precisely, there are $n-2$ degrees of freedom in~$R$ that can affect~$F$. Indeed, we could have assumed $R_{n}=0$ without loss of generality: subtracting a constant $c$ from all the $R_{k}$ will change $u$ into $u-c$ and $g$ into $g-c$ whereas the equilibrium distribution~$F$ is unchanged. Moreover, one can normalize the average (or the total) reward: replacing $R$ by $\lambda R$ for $\lambda>0$ changes $u$ into $\lambda u$ but leaves $F$ invariant.

(b) We have assumed that agents are risk-neutral wrt.\ the reward. This entails no loss of generality:
if agents optimize a utility function $U$ of the reward, we can treat $\tilde{R}_{k}:=U(R_k)$ as an auxiliary reward and agents as risk-neutral wrt.~$\tilde{R}$.

(c) As $F$ is atomless, ties and bankruptcies almost-surely do not occur.
\end{remark}

\section{Reward Design}\label{se:design}

We now study how the reward scheme influences the equilibrium stopping distribution and thus the players' levels of stopping, also called their \emph{performance} in what follows. While players only care about their rank, a principal interested in the performance of one or more players may optimize the reward scheme such as to induce a desirable performance. As above, rewards are fixed at the initial time and depend only on the final ranking.
Throughout, we normalize $R_{n}=0$ and vary $R_{1},\dots,R_{n-1}$ while keeping the total reward $\sum_{i=1}^n R_i=1$ constant; cf.\ Remark~\ref{rk:valuesUsed}\,(a). The standing assumption $\mu< \bar\mu$, cf.~\eqref{eq:mubar}, is in force for all reward schemes under discussion. This assumption is most stringent for the winner-takes-all scheme  ($R_{i}=0$ for $i>1$), where it reads 
\begin{equation}\label{eq:standAssumptStringent}
\mu<\frac{\sigma^{2}}{2x_{0}}\log\left(\frac{n}{n-1}\right).
\end{equation}

We identify two notions that are crucial for this discussion. First, the Lorenz order, which is a well-known measure of inequality in economics~\cite{ArnoldSarabia.18}.
Given two reward vectors $R$ and $\tR$ with the same total reward, $\tR$ exhibits \emph{less inequality than~$R$ in Lorenz order,} or
\[ \tR \leq_{L} R,\qquad \mbox{if}\qquad \sum_{i=1}^k \tR_i \leq \sum_{i=1}^k R_i \qforq k=1,\ldots, n.\]
Among all normalized reward vectors, the winner-takes-all scheme is the largest in Lorenz order whereas the uniform reward ($R_{1}=\dots=R_{n-1}$) is the smallest.
The upper bound $\bar{x}_{R}$ of the support of the equilibrium distribution $F$ corresponding to~$R$, cf.~\eqref{eq:xbar}, is increasing in~$R_{1}$. Hence, $\tR \leq_{L} R$ implies $\bar{x}_{R}\geq\bar{x}_{\tR}$, so that~$F$ and~$\tF$ (corresponding to~$\tR$)
are both concentrated on $(0,\bar{x}_{R})$.

The second notion refers to the  equilibrium distribution.
Given two cdf~$F$ and~$\tF$, we say that~$\tF$ is \emph{strictly single crossing wrt.~$F$} if there are $a<x_{1}<b$ with $F(a)=\tF(a)=0$ and $F(b)=\tF(b)=1$ as well as
\[
\tF<F \quad\mbox{on}\quad (a,x_{1}) \qquad\mbox{and}\qquad\tF>F  \quad\mbox{on}\quad (x_{1},b).
\]
Where it is useful to be more explicit, we say that the functions are strictly single crossing on $(a,b)$ with crossing point $x_{1}$. In  words, $\tF-F$ crosses the horizontal axis exactly once, in an increasing fashion, in an interval supporting both distributions. It means that as $F$ is transformed into $\tF$, a nontrivial part of the mass below $x_{1}$ is transported above $x_{1}$, thus reflecting an upward-mobility (in terms of level of stopping) inside the population of players.

Using the language of~\cite[Section~1.1]{DiamondStiglitz.74}, the economic interpretation of the following theorem is that a more unequal reward scheme induces a ``riskier'' equilibrium distribution.
In addition, it is also a tool for proving several of the results below.

\begin{figure}[h]
\centering
\includegraphics[height=7cm]{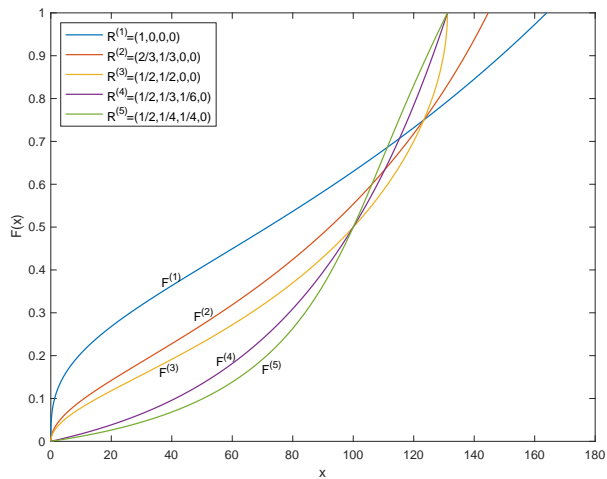}
\caption{Single crossing property of the equilibrium cdf $F^{(i)}$ corresponding to rewards $R^{(1)}\ge_L  \cdots \ge_L R^{(5)}$. Here $\mu=-0.01$, $\sigma=1$ and $x_0=100$. For $i=1,2,3$, the schemes $R^{(i)}$ only differ in the first two ranks and then $F^{(i)}$ intersect at a common point. Similarly for $i=3,4,5$, where the schemes differ in the second and third ranks. The distributions for $i=3,4,5$ have the same support; cf.\ \eqref{eq:xbar}.} 
\label{fig:F}
\end{figure}

\begin{theorem}\label{thm:single-crossing}
Let $R, \tR$ be distinct reward vectors and $F, \tF$ the corresponding equilibrium distributions. If $\tR \leq_{L} R$, then $\tF$ is strictly single crossing wrt.\ $F$.
\end{theorem}

\begin{proof}
Following Hardy, Littlewood and P\'olya (see \cite{MarshallOlkin.11}), the first step is to observe the result in the special case when the rewards differ only at two ranks: Fix $1\leq i < j <n$ and consider reward vectors $R,R^{\delta}$ where $R_{j}^{\delta}=R_{j}+\delta$ and $R_{i}^{\delta}=R_{i}-\delta$ and $R_{k}^{\delta}=R_{k}$ for $k\neq i,j$.
Let $F, F_\delta$ be the corresponding equilibrium distributions. Then for $\delta>0$, $F_\delta$ is strictly single crossing with respect to $F$ on $(0,\bar x_{F})$. Indeed, $\delta\mapsto F_{\delta}(x)$ is strictly decreasing for $x\in(0, x_1)$ and strictly increasing in $\delta$ for $x\in (x_{1},\bar x_{F})$, for a suitable~$x_{1}$. This can be shown by direct arguments, or one may combine the result of \cite[Lemma~9]{FangNoe.16} for capacity-constrained games with the transformation mentioned in the proof of Proposition~\ref{pr:equilibrium}.

Second, we observe that the change from $R$ to $\tR$ can be decomposed into a finite sequence $R^{(0)}, \ldots, R^{(N)}$ of such two-rank transformations, where $R^{(0)}=R$ and $R^{(N)}=\tR$. This is easily seen by induction (see \cite[Lemma~B.1, p.\,32]{MarshallOlkin.11} for a detailed proof). 
If the single crossing property were transitive, Theorem~\ref{thm:single-crossing} would be a direct consequence. It is not transitive, of course---but a careful argument is nevertheless successful.

Let $F_k$ be the equilibrium distribution induced by $R^{(k)}$. By the above, $F_k$ is strictly single crossing with respect to $F_{k-1}$. Let $x_{k}$ denote the crossing point,
$x_{\min}:=\min_{1\le k\le N} x_k$ and $x_{\max}:=\max_{1\le k\le N} x_k$, then $0<x_{\min}\le x_{\max}<\bar x_R$. For $x\in (0,x_{\min})$, the pairwise strict single crossing property implies $F_k(x)<F_{k-1}(x)$ for all~$k$,  hence $\tF(x)< F(x)$. A similar argument shows that $\tF(x)> F(x)$ for $x\in (x_{\max}, \bar x_R)$. Thus, by continuity, $\tF-F$ must cross zero from below at least once in $(x_{\min}, x_{\max})\subset (0, \bar x_R)$.

It remains to show that the zero of $\tF-F$ in $(0, \bar x_R)$ is unique. To this end, let $x_0\in(0,\bar x_R)$ be a zero of $\tF-F$ and $y_0=F(x_0)=\tF(x_0)$. As $F$ has a positive density on $(0,\bar x_R)$, it suffices to show the uniqueness of $y_0$. 
Note that $\tF(x_0)=F(x_0)<F(\bar x_R)=1$ implies $x_0< \bar x_{\tR}$. Since $R$ and $\tR$ have the same average and $R_1\ge \tR_1$, we see that $g(F(x))=u(x)=\tilde u(x)=\tilde g(\tF(x))$ on $[0,\bar x_{\tR}]$.  Setting $x=x_0$ yields $(\tilde g-g)(y_0)=0$; that is, $y_0$ must be a zero of $\tilde g-g$ in $(0,1)$.

Write $\tR-R=\sum_{(i,j)}\delta_{i,j} (e_j-e_i)$ where $e_{i}$ is the $i$-th basis vector and each term in the finite sum represents an inequality-reducing transformation changing the reward at two ranks: the amount $\delta_{i,j}>0$ is moved from the $i$-th place to the $j$-th place, where $i<j$. Let $P_k(y)$ be the probability of winning rank $k$ at location $y\in [0,1]$ if $(n-1)$ other random variables are i.i.d.\ and uniform on $[0,1]$.
Then
\begin{align*}
(\tilde g-g)(y)&=\sum_{k=1}^n (\tR_k-R_k) P_k(y)=\sum_{(i,j)} \delta_{i,j} (P_j(y)-P_i(y))\\
&=\sum_{(i,j)} \delta_{i,j} \left[ {{n-1}\choose{j-1}} y^{n-j}(1-y)^{j-1}-{{n-1}\choose{i-1}} y^{n-i}(1-y)^{i-1}\right]\\
&=\sum_{(i,j)} \delta_{i,j} y^{n-j}(1-y)^{i-1}\left[ {{n-1}\choose{j-1}} (1-y)^{j-i}-{{n-1}\choose{i-1}} y^{j-i}\right].
\end{align*}
Writing $G_{i,j}(y)$ for the expression in square brackets,
$(\tilde g-g)(y_0)=0$ and $0<y_0<1$ imply
$\sum_{i,j} \delta_{i,j} \frac{(1-y_0)^{i}}{y_0^j} G_{i,j}(y_0)=0.$
Both $G_{i,j}(y)$ and $(1-y)^i/y^j$ are strictly decreasing on $(0,1)$. Together with $\delta_{i,j}>0$, we conclude that $y_0$ is unique.
\end{proof}

\subsection{Performance of an Average Player}

Suppose a principal derives utility from the individual agent performance $X_\tau$ according to a utility function~$\phi$, then the expected utility in  equilibrium is
\[E[\phi(X_\tau)]=\int_0^\infty \phi(x)dF(x).\] 
We recall the scale function $h$ defined in~\eqref{eq:defh}, a smooth function with $h'>0$ that is concave for $\mu\geq0$ and convex for $\mu\leq0$.

\begin{theorem}\label{thm:utility}
Let $R, \tR$ be distinct reward vectors with $\tR \leq_{L} R$
and $F, \tF$ the corresponding equilibrium distributions.
Let $\phi: \R_+\rightarrow \R$ be an increasing, absolutely continuous function.
\begin{itemize}
\item[(i)] If $\phi'/h'$ is increasing on $(0,\bar x_R)$, then $\int_0^\infty \phi(x)d\tF(x)\le \int_0^\infty \phi(x)dF(x)$.
\item[(ii)] If $\phi'/h'$ is decreasing on $(0,\bar x_R)$, then $\int_0^\infty \phi(x)d\tF(x)\ge \int_0^\infty \phi(x)dF(x)$.
\end{itemize}
The inequalities are strict unless $\phi=ah+b$ for some constants $a,b$. 
\end{theorem}
\begin{proof}
(i) Integration by parts yields
\begin{align*}
\int_0^\infty \phi(x) d(\tF-F)(x)&=-\int_0^{\bar x_R} (\tF-F)(x)\phi'(x)dx.
\end{align*}
By Theorem~\ref{thm:single-crossing}, $\tF$ is strictly single crossing wrt.\ $F$ with some crossing point $x_1\in(0, \bar x_R)$. As $\phi'/h'$ is increasing and $h'>0$, 
\[\int_0^{\bar x_R} (\tF-F)(x)h'(x) \frac{\phi'(x)}{h'(x)}dx\ge  \frac{\phi'(x_1)}{h'(x_1)} \int_{0}^{\bar x_R} (\tF-F)(x)h'(x)dx.\]
Another integration by parts gives
\begin{align*}
\int_{0}^{\bar x_R} (\tF-F)(x)h'(x)dx&=(\tF-F)(x) h(x)\bigg|_{x=0}^{\bar x_R}-\int_0^{\bar x_R}h(x)d(\tF-F)(x)
=0,
\end{align*}
where the last equality holds by Lemma~\ref{le:solFexists}. Combining the above displays, we have
$\int_0^\infty \phi(x) d(\tF-F)(x)\le 0$, and the inequality is strict unless $\phi'/h'\equiv \phi'(x_1)/h'(x_1)$ a.e. 
The proof of~(ii) is analogous.
\end{proof}

Specializing to risk-averse and risk-seeking utility functions, we obtain the following.

\begin{corollary}\label{cor:utility}
Let $R, \tR, F, \tF, \phi$ be as in Theorem~\ref{thm:utility}. 
\begin{itemize}
\item[(i)] If $\phi$ is convex and $\mu\ge 0$, then $\int_0^\infty \phi(x)d\tF(x)\le \int_0^\infty \phi(x)dF(x)$.
\item[(ii)] If $\phi$ is concave and $\mu\le 0$, then $\int_0^\infty \phi(x)d\tF(x)\ge\int_0^\infty \phi(x)dF(x)$.
\end{itemize}
If $\mu\neq 0$ and $\phi$ is not constant, the asserted inequality is strict.
\end{corollary}

\begin{proof}
This follows from the concavity/convexity of $h$ and Theorem~\ref{thm:utility}.
\end{proof}

Intuitively, the reward allocation induces a ``risk preference'' in agents. This comparison can be motivated via Remark~\ref{rk:valuesUsed}\,(b): Starting from a reward $R$, consider a concave increasing function $U$ and $\tR:=U(R)$. By an affine normalization of $U$ we may assume that $\tR$ is again a normalized reward. It is easy to see that $\tR \leq_{L}R$; cf.\ \cite[Proposition~B.2, p.\,188]{MarshallOlkin.11}. That is, risk-neutral players (as we have assumed) with reward allocation $\tR$ are equivalent to risk-averse players with allocation~$R$. 
Conversely, the more unequal the reward, the more risk-seeking agents become, staying longer in the game to gamble for a high performance (see also Corollary~\ref{co:average-time} below).
 
Corollary~\ref{cor:utility} shows that the principal should align agents' risk preferences with her own, provided that the market condition $\mu$ is not too strong a counter force. A negative drift reinforces a risk-averse principal's preference for agents to stop early, to reduce both variance and expected losses due to the drift, whereas a positive drift reinforces the preference to gamble and profit from the drift.
If the principal's preferences and the market condition are opposed, the trade-off results in an ambiguous comparison, as shown by the following example.

\begin{example}[Risk-averse principal in a bull market]\label{ex:principal-bull-market}
Let $\mu>0$ and $\phi(x)=-\frac{1}{\gamma}e^{-\gamma x}$ where $\gamma>0$. 
Then
$\frac{\phi'(x)}{h'(x)}=\frac{\sigma^2}{2\mu} (1-\exp (\frac{-2\mu x_{0}}{\sigma^{2}}))  \exp((\frac{2\mu }{\sigma^{2}}-\gamma)x)$;
thus $\phi'/h'$ is strictly increasing if $2\mu/\sigma^2>\gamma$, strictly decreasing if $2\mu/\sigma^2<\gamma$, and constant if $2\mu/\sigma^2=\gamma$. As a result, reward inequality is preferred for small values of the risk aversion $\gamma$ whereas equality is preferred for large values.
\end{example}

Clearly Corollary~\ref{cor:utility} can be used  to analyze the optimal reward scheme for the principal. We only state the result for linear utility.

\begin{corollary}\label{co:total-performance}
The expected performance $E[X_\tau]$ is strictly increasing wrt.\ the Lorenz order of the reward scheme when $\mu>0$, and strictly decreasing when $\mu<0$. In particular, $E[X_\tau]$ is maximized by the winner-takes-all scheme when $\mu>0$ and by the uniform reward when $\mu<0$. For $\mu=0$, the expected performance is independent of the reward.
\end{corollary} 
 
\begin{proof}
The result follows from Corollary~\ref{cor:utility} with $\phi(x)=x$ after noting that uniform and winner-takes-all are, respectively, the unique minimum and maximum elements wrt.\  Lorenz order among all normalized reward schemes. 
\end{proof}

\begin{figure}[tbh]
\centering
\includegraphics[height=6.2cm, width=9cm]{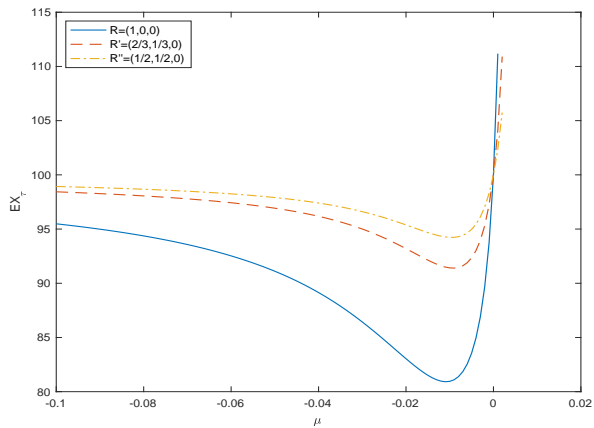}
\caption{Average performance $E[X_\tau]$ as a function of drift $\mu$ for three different reward schemes $R''\leq_{L}R'\leq_{L}R$. Here $x_0=100$ and $\sigma=1$. The vertical asymptotes of the three curves associated with $R$, $R'$ and $R''$ are at $\bar \mu=0.002$, 0.0035 and 0.0055, respectively.}
\label{fig:mu}
\end{figure}

See Figure~\ref{fig:mu} for numerical examples illustrating Corollary~\ref{co:total-performance}. The figure also shows that, similarly as in~\cite{SeelStrack.13}, the largest losses occur for an intermediate value of negative drift~$\mu$. While the corresponding~$\mu$ varies only slightly with the reward scheme, the losses for winner-takes-all are substantially larger than for the schemes with lower inequality.

As alluded above, we can show that higher reward inequality implies that players gamble longer, in line with the interpretation given below Corollary~\ref{cor:utility}. As players only care about their relative ranking and not about the absolute performance, it is natural that the sign of the drift does not appear in this  result. 

\begin{corollary}\label{co:average-time}
The expected duration $E[\tau]$ of play is monotone increasing wrt.\ the Lorenz order of the reward scheme. In particular, it is maximized by the winner-takes-all and minimized by the uniform scheme.
\end{corollary}

\begin{proof}
When $\mu\neq 0$, optional sampling yields $E[X_{\tau}]=x_{0}+\mu E[\tau]$. The result then follows from Corollary~\ref{co:total-performance}.
When $\mu=0$, we apply the optional sampling theorem to the martingale $X_t^2-\sigma^2 t$ and use Corollary~\ref{cor:utility} with $\phi(x)=x^2$.
\end{proof}

\begin{remark}\label{rk:stochOrder}
  If $\mu\le 0$, then $\tF$ dominates $F$ in second stochastic order; i.e., $\int_0^y (\tF(x)-F(x))dx \le 0$ for all $y\ge 0$. Indeed, this order is alternately characterized through integrals of increasing concave functions, so that the claim is a reformulation of Corollary~\ref{cor:utility}\,(ii). The interpretation is as above: a more equitable reward makes players prefer less variance and stop earlier, hence suffer less from the negative drift and achieve a higher performance in equilibrium.
  
  For $\mu>0$, Example~\ref{ex:principal-bull-market} shows that $\tF$ and $F$ cannot be ordered in this sense, as that would imply that the principal's preference is the same for all positive risk aversion parameters.
  
  For $\mu=0$, the game is equivalent to the capacity-constrained game of \cite{FangNoe.16} and the second stochastic dominance is shown in \cite[Proposition~5]{FangNoe.16}. For $\mu>0$, the order is not preserved by the transformation mentioned in the proof of Proposition~\ref{pr:equilibrium}, as evidenced by the aforementioned example.
\end{remark} 

\subsection{Performance of the First Rank}

Next, we study the problem of a principal aiming to maximize the expected equilibrium performance of the first-ranked player,
\[E\left[\max_{i=1,\ldots,n} X_{\tau_i}\right]=n \int_0^{\bar x} x F(x)^{n-1} dF(x).
\]
In contrast to the preceding subsection, this constitutes a nonlinear functional of the equilibrium distribution, and we obtain a result that is independent of the drift (even though the proofs differ depending on the sign). The first rank naturally incorporates an upwards bias relative to the average performance, and the difference increases with the volatility. For positive drift, this strongly suggests that the principal will profit from gambling and thus should encourage a long duration of the game. More surprisingly, the profit from volatility turns out to be more important than any losses that may occur due to a negative drift, so that reward inequality  is preferred in any market condition.
 
\begin{theorem}\label{th:1st-rank}
The expected performance $E[\max_i X_{\tau_i}]$ of the first-ranked player is strictly increasing wrt.\ the Lorenz order of the reward scheme. In particular, the winner-takes-all scheme is the unique maximizer.
\end{theorem}

The following lemma is required for the proof. For later use, we state it for the $k$-th rank rather than just the first rank.

\begin{lemma}\label{le:kthOrderStats}
  Let $R$ be a reward scheme and $F$ the associated equilibrium distribution. Let $(Y_{i})_{1\leq i\leq  n}$ be i.i.d.\ with distribution $F$ and denote by $Y^{(k)}$ the $k$-th reverse order statistic (the $k$-th largest value), where $1\leq k\leq n-1$.
  \begin{enumerate}
  \item
  If $\mu=0$, then
  \begin{equation}\label{eq:kthOrderStats}
    E[Y^{(k)}] = nx_{0} \frac{n!}{(2n-1)!} {{n-1}\choose{k-1}} \sum_{l=1}^{n} R_{l} 
    \phi(k,l), \quad \mbox{where}
  \end{equation}
  \begin{equation}\label{eq:phiDefn}
  \phi(k,l):=\frac{(2n-k-l)! (k+l-2)!}{(n-l)!(l-1)!}.
  \end{equation}
  
  \item
  If $\mu\neq 0$, then setting $A=\frac{-2\mu}{\sigma^{2}}$ and $B=\exp(Ax_{0})-1$, 
    $$
    E[Y^{(k)}] = n {{n-1}\choose{k-1}}  A^{-1} \int_{0}^{1}\log[nBg(y)+1)] y^{n-k}(1-y)^{k-1} dy.
  $$
  \end{enumerate}
    In particular, $E[Y^{(k)}]$ is strictly concave with respect to $R$ for $\mu<0$, strictly convex for $\mu>0$, and linear for $\mu=0$.
\end{lemma} 

\begin{proof}
  Recall that $F$ is strictly increasing on $[0,\bar x]$, hence admits an inverse $q:=F^{-1}$. Clearly
  \begin{align}\label{eq:kthOrderGeneral}
    E[Y^{(k)}] 
    &= n {{n-1}\choose{k-1}} \int_{0}^{\bar x} x F(x)^{n-k}(1-F(x))^{k-1} dF(x) \nonumber\\
    &= n {{n-1}\choose{k-1}} \int_{0}^{1} q(y) y^{n-k}(1-y)^{k-1} dy.
  \end{align}
  In view of $u(x)=g(F(x))$, we have $q(y) = u^{-1}(g (y))$ for $0\leq y\leq 1$.

  (i) Let $\mu=0$. As $R$ is normalized with $\bar R=1/n$, we obtain $u(x)=\frac{x}{nx_{0}}$ and $u^{-1}(y)=nx_{0}y$. As a result,
  $
    q(y) = nx_{0} g(y),
  $
  and then by~\eqref{eq:kthOrderGeneral},
  \begin{align*}
    E[Y^{(k)}] 
    &= n^{2}x_{0} {{n-1}\choose{k-1}} \int_{0}^{1} g(y) y^{n-k}(1-y)^{k-1} dy\\
    &= n^{2}x_{0} {{n-1}\choose{k-1}} \sum_{l=1}^{n} R_{l} {{n-1}\choose{l-1}} \int_{0}^{1} y^{2n-k-l} (1-y)^{k+l-2} dy.
  \end{align*}
  To compute this expression, we note that
  \begin{align*}
   \int_{0}^{1} y^{2n-k-l} (1-y)^{k+l-2} dy
    &=\BetaFun(2n-k-l+1,k+l-1)\\
    &=\frac{(2n-k-l)! (k+l-2)!}{(2n-1)!}
  \end{align*}
  where we have used that the Beta function $\BetaFun(x,y)=\int_{0}^{1} t^{x-1}(1-t)^{y-1}  dt$ satisfies the relation 
  $\BetaFun(x,y)=\Gamma(x)\Gamma(y)/\Gamma(x+y)$ with the Gamma function.
  As a result,
    \begin{align*}
    E[Y^{(k)}] 
    &= n^{2}x_{0} {{n-1}\choose{k-1}} \sum_{l=1}^{n} R_{l} {{n-1}\choose{l-1}}\frac{(2n-k-l)! (k+l-2)!}{(2n-1)!}\\
    & = nx_{0} \frac{n!}{(2n-1)!} {{n-1}\choose{k-1}} \sum_{l=1}^{n} R_{l} \frac{(2n-k-l)! (k+l-2)!}{(n-l)!(l-1)!}.
  \end{align*}
  
  (ii) Let $\mu\neq0$. Note 
  $h(x)=\frac{\exp(Ax)-1}{B}$, hence $h^{-1}(z)=A^{-1} \log(Bz+1)$.
  As $u(x)=\frac1n h(x)$ for $x\le \bar x$, we have $u^{-1}(z)=h^{-1}(nz)$; i.e.,
$$
  q(y)=u^{-1}(g(y))=A^{-1} \log(nBg(y)+1).
$$ 
This expression is well defined due to~\eqref{eq:standAssumptStringent}. In view of~\eqref{eq:kthOrderGeneral}, the claim follows.
\end{proof}

\begin{proof}[Proof of Theorem~\ref{th:1st-rank}]
Let $\tilde{R}\leq_{L}R$ be two reward schemes and $\tilde{F},F$ the corresponding equilibria. By Theorem~\ref{thm:single-crossing}, $\tilde{F}$ is strictly single crossing wrt.~$F$. 

(i) Case $\mu\geq0$. 
We also have that $F^{-1}$ is strictly single crossing with respect to $\tilde{F}^{-1}$ on $(0,1)$. Let $y_0$ be the crossing point, then
\[\int_0^1 (F^{-1}(y) -\tilde{F}^{-1}(y))y^{n-1}dy> y_0^{n-1} \int_0^1 (F^{-1}(y) -\tilde{F}^{-1}(y))dy\ge 0,\]
where the last inequality is due to Corollary~\ref{co:total-performance} and $\mu\geq0$.

(ii) Case $\mu<0$. For $\lambda\in[0,1]$, we define (cf.\ Lemma~\ref{le:solFexists})
\[\varphi(\lambda):= E[Y^{(1)}_\lambda]=  n A^{-1} \int_{0}^{1}\log[nB(\lambda \tilde{g}(y)+(1-\lambda)g(y))+1] y^{n-1} dy\] 
and show that $\varphi$ attains its unique maximum at $\lambda=0$. As $\varphi$ is strictly concave
(Lemma~\ref{le:kthOrderStats}), it suffices to show that the right derivative $\varphi'(0+)\le 0$.
Indeed,
\[\varphi'(0+)= nA^{-1} \int_{0}^{1}\frac{nB(\tilde{g}-g)(y) y^{n-1} }{ nBg(y)+1}   dy.
\]
As $\mu<0$, we have $B>0$ and one checks that the factor
\[\frac{y^{n-1}}{nBg(y)+1}=\left[\sum_{\ell=1}^n \left(nB R_\ell+1\right) {{n-1}\choose{\ell-1}} \left(\frac{1-y}{y}\right)^{\ell-1}\right]^{-1}\]
is increasing in $y$.
In view of $\tilde{R} \leq_{L}R$, $g$ is strictly single-crossing with respect to $\tilde{g}$ on $(0,1)$. Finally, $\int_{0}^{1} \tilde{g}(y) dy=\bar{R}=\int_{0}^{1} g(y) dy$. Together, these three facts imply that $\varphi'(0+)\leq0$.
\end{proof}

%%%%%%%%%%%%%%%%%%%%%%%%%%%%%%%%%%%%%%%%
\subsection{Performance of the $k$-th Rank}

We consider a principal maximizing the expected performance of the $k$-th ranked player, where $1\leq k\leq n-1$. This problem is more involved that the first rank: if $k$ is close to 1 (relative to $n/2$), we may expect to see similar effects as for the first rank, but clearly the profits from volatility are weaker. A first guess may be that the principal should maximize the reward at the $k$-th rank in order to maximize $k$-th rank performance. While this is not always true, the following reward schemes nevertheless play a special role.

\begin{definition}\label{de:cutoff}
   For $1\leq  j \leq n-1$, the reward scheme $R^{j}=(R^{j}_{1},\dots,R^{j}_{n})$ with 
$$
   R^{j}_{i}=1/j,\quad i\leq j \qandq R^{j}_{i}=0,\quad i> j
$$
is called the cut-off at $j$.
\end{definition}

In words, $R^{j}$ distributes the total reward uniformly over the first $j$ ranks. This scheme maximizes the reward at the $j$-th rank. The winner-takes-all scheme $R^{1}$ and the uniform scheme $R^{n-1}$ are special cases.

We first focus on the case of zero drift which allows for the most detailed analysis. When $k=1$, we have seen in Theorem~\ref{th:1st-rank} that the winner-takes-all reward is optimal. The next result shows that the guess also holds for the second rank: it is optimal to reward the first two ranks equally, and give zero reward to the subsequent ranks. However, this does not extend to higher target ranks~$k\geq3$. While a cut-off reward is still optimal, it can be beneficial to extend the cut-off point beyond~$k$. The analytic description uses the function $\phi$ of~\eqref{eq:phiDefn}.

\begin{proposition}\label{pr:zeroDriftCutoff}
  Let $\mu=0$. Then the unique normalized reward scheme maximizing the expected performance $E[Y^{(k)}]$ of the $k$-th rank is the cut-off at $k_{*}$, where
  \begin{equation}\label{eq:kstar}
    k_{*}=\max\left\{j\geq k:  \phi(k,j)\geq \frac{1}{j-1}\sum_{l=1}^{j-1} \phi(k,l)\right\}.
  \end{equation} 
  In particular, the winner-takes-all scheme is optimal for $k=1$ and the cut-off at 2 is optimal for $k=2$. For $k\geq3$, it may happen that $k_{*}>k$. For instance, for $n=5$ and $k=3$, the cut-off at $k_{*}=4$ is optimal; for $n=10$ and $k=5$, the cut-off at $k_{*}=7$ is optimal (cf.\ Figure~\ref{fig:kstar}).
\end{proposition}

\begin{figure}[bth]
\centering
\includegraphics[width=\textwidth]{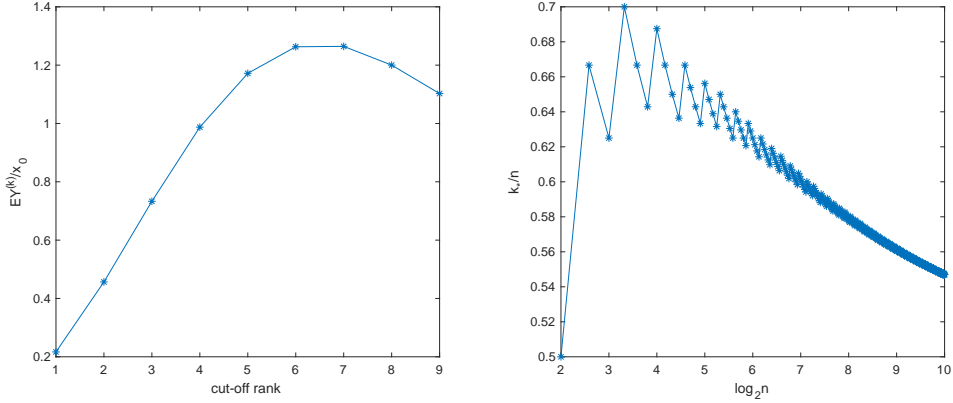}
\caption{Illustration of Proposition~\ref{pr:zeroDriftCutoff}. The left panel shows the $k$-th rank performance
for all cut-off schemes when $n=10$ and $k=5$; the best performance is attained at $k_*=7$. The right panel shows the optimal cut-off ratio $k_*(n)/n$ when the target rank $k$ varies with~$n$, chosen such that $k/n=\alpha:=1/2$ is constant. The behavior for finite~$n$ is rather complex but suggests a simplification in the limit $n\to\infty$, which has motivated the study of the limiting mean field game in a companion paper~\cite{NutzZhang.21b}.}
\label{fig:kstar}
\end{figure}

\begin{proof}
We have 
$
  \frac{\phi(k,l+1)}{\phi(k,l)} = \frac{(k+l-1)(n-l)}{(2n-k-l) l}.
$
Noting that
$$ 
  (k+l-1)(n-l) - (2n-k-l) l = n(k-l)+ l -n
$$
is $<0$ if $l\geq k$ and $>0$ if $l< k$, we see that 
$\frac{\phi(k,l+1)}{\phi(k,l)}<1$ if $l\geq k$ and $\frac{\phi(k,l+1)}{\phi(k,l)}>1$ if $l< k$. That is, we have
$$
  \phi(k,1)<\phi(k,2)<\cdots <\phi(k,k-1)< \phi(k,k) > \phi(k,k+1) > \cdots \phi(k,n-1)
$$
and in particular $\phi(k,k)$ is a maximum. In view of~\eqref{eq:kthOrderStats}, we conclude that an optimal reward scheme must pay equal rewards to ranks $j=1,\dots,k_{*}$. 
For $k=1$ it follows directly that $k_{*}=1$. For $k=2$ we note that $\phi(2,1)>\phi(2,3)$ holds for all $n$, which of course implies that $\frac12 [\phi(2,1)+\phi(2,2)]>\phi(2,3)$. The further examples are verified by direct calculation.
\end{proof}

%%%%%%%%%%%%%%%%%%%%%%%%%%%%%%%%%%%%%%%%%

We now turn to the case of non-zero drift, where our result is less detailed. The number $k_{*}$ is defined in~\eqref{eq:kstar}.

\begin{proposition}\label{pr:rank-k-nonzero-drift}
  If $\mu>0$, the expected performance $E[Y^{(k)}]$ is maximized by a cut-off at $j$ for some $j\le k_\ast$.
If $\mu<0$, the optimal reward scheme pays equal amounts to ranks 1 through $k_{*}$.
\end{proposition}

\begin{remark}
(a) For $\mu<0$, the optimizer need \emph{not} be a cut-off scheme. That is, in addition to the equal amounts mentioned in the proposition, smaller amounts may be paid to lower ranks. 
For instance, let $\mu=-0.5$, $\sigma=1$, $x_0=1$ and $(n,k)=(5,2)$. Then $k_*=2$ and numerical experiments show that the optimal reward scheme is given by $(0.416, 0.416, 0.168, 0, 0)$, which is not a cut-off scheme.

(b) For $\mu>0$, we conjecture that the optimal $j$ satisfies $k\leq j\leq k_{*}$. Both inequalities may be strict.
As an example, let $\mu=0.05$, $\sigma=1$, $x_0=1$ and $(n,k)=(10,5)$. In this case, $\mu<\bar\mu$ is satisfied for all rewards. We have $k_*=7$ and numerical experiments show that the cut-off at $j=6$ is optimal.
\end{remark}

\begin{proof}[Proof of Proposition~\ref{pr:rank-k-nonzero-drift}]
The cut-off schemes $(R^{i})_{i=1,\dots,n-1}$ are the extreme points of the compact, convex set of normalized reward schemes. Any normalized reward
$R$ can be uniquely expressed as a convex combination $R=\sum_{i=1}^{n-1}\lambda_i R^{i}$ where $\lambda=(\lambda_{1},\dots,\lambda_{n-1})$ is an element of the unit simplex $\Delta\subset \R^{n-1}$. Introducing the function $g^{i}$ associated with $R^{i}$ as in Lemma~\ref{le:solFexists},
\[g^i(y):=\sum_{l=1}^n R^i_l{{n-1}\choose{l-1}}y^{n-l}(1-y)^{l-1}=\sum_{l=1}^i \frac{1}{i}{{n-1}\choose{l-1}}y^{n-l}(1-y)^{l-1},
\]
we can rewrite the optimization over normalized reward schemes as 
\[\sup_{\lambda\in \Delta} \, n {{n-1}\choose{k-1}}  A^{-1} \int_{0}^{1}\log\left[nB \sum_{i=1}^{n-1}\lambda_i g^i(y) +1\right] y^{n-k}(1-y)^{k-1} dy.\]
Dropping a positive factor for brevity, we thus seek to maximize
\begin{equation}\label{eq:Jlam}
J(\lambda):= A^{-1} \int_{0}^{1}\log\left[nB \sum_{i=1}^{n-1}\lambda_i g^i(y) +1\right] y^{n-k}(1-y)^{k-1} dy
\end{equation}
over $\Delta$. This is a strictly convex, continuous function for $\mu>0$, showing that any optimizer must be an extreme point. Whereas for $\mu<0$, $J$ is strictly concave, showing that the optimizer is unique (and explaining why the solution may well be an interior point rather than a cut-off scheme).

(i) Let $\mu>0$, so that $A, B<0$. Fix $k_\ast<j<n$,  then $g^{k_\ast}$ is strictly single-crossing wrt.\ $g^j$ with some crossing point $y_0\in (0,1)$. Writing  $e_{i}$ for the $i$-th basis vector in $\R^{n-1}$, and using also that $x\le e^x-1$, with equality only for $x=1$, the crossing property implies 
\begin{align*}
J(e_{k_\ast})-J(e_j)&=A^{-1}\int_0^1 \log\left(\frac{nB g^{k_\ast}(y)+1}{nB g^j(y)+1}\right) y^{n-k}(1-y)^{k-1}dy\\
&> A^{-1}\int_0^1 \left(\frac{nB g^{k_\ast}(y)+1}{nB g^j(y)+1}-1\right) y^{n-k}(1-y)^{k-1}dy\\
&= \frac{nB}{A}\int_0^1 \frac{g^{k_\ast}(y)-g^j(y)}{nB g^j(y)+1} y^{n-k}(1-y)^{k-1}dy\\
& \ge  \frac{nB}{A(nB g^j(y_0)+1)}\int_0^1 (g^{k_\ast}(y)-g^j(y)) y^{n-k}(1-y)^{k-1}dy.
\end{align*}
Moreover,
\begin{align*}
&\int_0^1 (g^{k_\ast}(y)-g^j(y)) y^{n-k}(1-y)^{k-1}dy\\
&= \sum_{l=1}^{n}\left(\frac{1_{l\le k_\ast}}{k_\ast}-\frac{1_{l\le j}}{j} \right){{n-1}\choose{l-1}}\int_0^1 y^{2n-k-l}(1-y)^{k+l-2}dy\\
&= \sum_{l=1}^{n}\left(\frac{1_{l\le k_\ast}}{k_\ast}-\frac{1_{l\le j}}{j} \right){{n-1}\choose{l-1}}\frac{(2n-k-l)! (k+l-2)!}{(2n-1)!}\\
&= \frac{(n-1)!}{(2n-1)!}\sum_{l=1}^{n}\left(\frac{1_{l\le k_\ast}}{k_\ast}-\frac{1_{l\le j}}{j} \right)\phi(k,l)\\
&= \frac{(n-1)!}{(2n-1)!}\left(\frac{1}{k_\ast} \sum_{l=1}^{k_\ast} \phi(k,l)-\frac{1}{j} \sum_{l=1}^{j} \phi(k,l)\right).
\end{align*}
The last expression is nonnegative by the definition of~$k_\ast$. Putting everything together, we have shown that $e_j$ is strictly suboptimal and the claim follows.

(ii) Let $\mu<0$, so that $A, B>0$. Let $\lambda^0\in\Delta$ be such that $\lambda^0_{i_0}>0$ for some $i_0<k_\ast$ and define $\lambda^1:=\lambda^0+\lambda^0_{i_0}(e_{k_\ast}-e_{i_0})$. 
To show that $\lambda^1$ is strictly better than $\lambda^0$, it suffices by concavity to show
$\frac{\partial}{\partial \theta} J(\lambda^\theta)\big|_{\theta=0}< 0,$
where $\lambda^\theta:=\theta\lambda^0+(1-\theta) \lambda^1$.
Indeed, we have
\begin{align*}
\frac{\partial}{\partial \theta} J(\lambda^\theta)\big|_{\theta=0}
&=\frac{nB}{A}\int_0^1 \frac{\sum_{i=1}^{n-1}(\lambda^0_i-\lambda^1_i)g^i(y)}{ nB\sum_{i=1}^{n-1}\lambda^1_i g^i(y) +1}  y^{n-k}(1-y)^{k-1} dy.
\end{align*}
Using the single-crossing property of 
$g^{i_0}$ with respect to $g^{k_\ast}$, the strict monotonicity of $nB\sum_{i=1}^{n-1}\lambda^1_i g^i(y) +1$, and the definition of $k_\ast$, we deduce that

\begin{align*}
\frac{\partial}{\partial \theta}  J(\lambda^\theta)\big|_{\theta=0} &
=\frac{nB}{A}\int_0^1 \frac{\lambda^0_{i_0}(g^{i_0}(y)-g^{k_\ast}(y))}{ nB\sum_{i=1}^{n-1}\lambda^1_i g^i(y) +1}  y^{n-k}(1-y)^{k-1} dy\\
&<C\int_0^1 \lambda^0_{i_0}(g^{i_0}(y)-g^{k_\ast}(y))  y^{n-k}(1-y)^{k-1} dy\\
&=C \int_0^1 \lambda^0_{i_0} \sum_{l=1}^{n}\left( \frac{1_{l\le i_0}}{i_0}-\frac{1_{l\le k_\ast}}{k_\ast}  \right){{n-1}\choose{l-1}}y^{2n-k-l}(1-y)^{k+l-2}dy\\
&=C\lambda^0_{i_0}  \sum_{l=1}^{n}\left( \frac{1_{l\le i_0}}{i_0}-\frac{1_{l\le k_\ast}}{k_\ast}  \right){{n-1}\choose{l-1}}\frac{(2n-k-l)! (k+l-2)!}{(2n-1)!}\\
&=C\left(\frac{1}{i_0} \sum_{l=1}^{i_0} \phi(k,l)-\frac{1}{k_\ast} \sum_{l=1}^{k_\ast} \phi(k,l)\right)\le 0,
\end{align*}
where $C$ is a positive constant that may vary from line to line. This shows that $J(\lambda^1)>J(\lambda^0)$. As a consequence, the optimal reward scheme must be a convex combination of $R^{k_{*}},\dots,R^{n-1}$.
\end{proof}

%%%%%%%%%%%%%%%%%%%%%%%%%%%%%%%%%%%%%%%%%%%%%%%%%%%%%%%%%%%%%%%%%
\bibliography{stochfin}
\bibliographystyle{plain}
%%%%%%%%%%%%%%%%%%%%%%%%%%%%%%%%%%%%%%%%%%%%%%%%%%%%%%%%%%%%%%%%%
\end{document}